\documentclass[12pt,reqno]{amsart}
\usepackage{amsmath,amssymb,amsfonts,amscd,latexsym,amsthm,mathrsfs,verbatim}
\usepackage[unicode]{hyperref}
\usepackage{hypbmsec}
\hypersetup{
colorlinks=true,
linkcolor=blue,
citecolor=blue,
urlcolor=blue,
filecolor=blue,
bookmarksnumbered=true,
pdfstartview=FitH,
pdfhighlight=/N
}
\newtheorem{theorem}{Theorem}

\theoremstyle{remark}

\newtheorem{example}{\bf Example}

\renewcommand{\d}{{\mathrm d}}
\renewcommand{\Re}{\operatorname{Re}}
\begin{document}

\title{Ramanujan-type formulae for $1/\pi$: The~art~of~translation}

\author{Jes\'us Guillera}
\address{Av.\ Ces\'areo Alierta, 31 esc.~izda 4$^\circ$--A, Zaragoza, SPAIN}
\email{jguillera@gmail.com}

\author{Wadim Zudilin}
\address{School of Mathematical and Physical Sciences,
The University of Newcastle, Callaghan, NSW 2308, AUSTRALIA}
\email{wzudilin@gmail.com}

\thanks{Work of the second author is supported by Australian Research Council grant DP110104419.}

\date{25 December 2012. \emph{Revised}: 19 February 2013}


\begin{abstract}
We outline an elementary method for proving numerical hypergeometric identities, in particular, Ramanujan-type identities for $1/\pi$.
The principal idea is using algebraic transformations of arithmetic hypergeometric series to translate non-singular points
into singular ones, where the required constants can be computed using asymptotic analysis.
\end{abstract}

\maketitle

\hbox to\hsize{\hfil
\hbox{\vbox{\hsize=12cm\small\footnotesize\par
\dots I have realised that a good translator deserves honour in our literary circles,
because he is not a craftsman, not copyist, but the artist. He does not photograph the original,
as was commonly believed then, but recreates it creatively. The text is the original material
for his complex and often inspired creativity. Translator is first of all talent.
To translate Balzac he needs to at least partially transform into Balzac,
to assimilate his temper, to get his enthusiasm, his poetic sense of life.

\begin{center}
\textsc{K.\,I.~Chukovsky}, {\em A High Art}
\end{center}
}}}

\section{Introduction}
\label{s1}

In what follows we keep the standard notation
$$
(s)_n=\frac{\Gamma(s+n)}{\Gamma(s)}
$$
for Pochhammer's symbol (also known as shifted factorial). When $n$ is a positive integer we have
$(s)_n=\prod_{k=0}^{n-1}(s+k)$; in particular, $(1)_n=n!$\,.

\smallskip
The quantity $\pi$ is linked to the mathematical legacy of Srinivasa Ramanujan in many different ways.
In his development of the theory of elliptic functions, Ramanujan came up \cite{Ra} with a bunch of computationally efficient
representations of $1/\pi$, like
\begin{align}
\label{rama-ex6}
\sum_{n=0}^{\infty}\frac{(\frac12)_n^3}{n!^3}\,(1+6n)\,\frac{1}{2^{2n}}
&=\frac{4}{\pi},
\\
\label{rama-64}
\sum_{n=0}^{\infty}\frac{(\frac12)_n^3}{n!^3}\,(5+42n)\,\frac{1}{2^{6n}}
&=\frac{16}{\pi},
\displaybreak[2]\\
\label{rama28}
\sum_{n=0}^{\infty}\frac{(\frac12)_n(\frac16)_n(\frac56)_n}{n!^3}\,(8+133n)\biggl(\frac{4}{85}\biggr)^{3n}
&=\frac{85\sqrt{255}}{54\pi},
\displaybreak[2]\\
\label{e00}
\sum_{n=0}^\infty\frac{(\frac12)_n(\frac14)_n(\frac34)_n}{n!^3}\,(3+28n)\biggl(-\frac1{48}\biggr)^n
&=\frac{16}{\pi\sqrt3},
\displaybreak[2]\\
\label{e03}
\sum_{n=0}^\infty\frac{(\frac12)_n(\frac14)_n(\frac34)_n}{n!^3}\,(1123+21460n)\biggl(-\frac1{882^2}\biggr)^n
&=\frac{4\cdot882}\pi,
\displaybreak[2]\\
\label{rama-10-1}
\sum_{n=0}^{\infty}\frac{(\frac12)_n(\frac14)_n(\frac34)_n}{n!^3}\,(1+10n)\,\frac{1}{3^{4n}}
&=\frac{9\sqrt2}{4 \pi},
\displaybreak[2]\\
\label{rama7-4}
\sum_{n=0}^{\infty}\frac{(\frac12)_n(\frac14)_n(\frac34)_n}{n!^3}\,(3+40n)\,\frac{1}{7^{4n}}
&=\frac{49\sqrt3}{9\pi},
\\
\label{e04}
\sum_{n=0}^\infty\frac{(\frac12)_n(\frac14)_n(\frac34)_n}{n!^3}\,(1103+26390n)\,\frac1{99^{4n}}
&=\frac{99^2}{2\pi\sqrt2};
\end{align}
these are equations (28), (29), (34), (36), (39), (41), (42) and (44) on the list in~\cite{Ra}. Ramanujan did not hide
his interest in computing $\pi$; his comment in~\cite{Ra} about identity~\eqref{e04} says
``The last series (44) is extremely rapidly convergent.''

The identities do not look hard. In spite of this, their first proofs were only obtained in the 1980s by the Borweins \cite{Bo}
and by the Chudnovskys~\cite{Chud}. A historical account of contemporary techniques for proving Ramanujan's and Ramanujan-type
formulae as well as their supercongruence relatives can be found in \cite{BaBeCh,Zu,Zu0}. The dominating method which, for example,
works for any formulae in~\eqref{rama-ex6}--\eqref{e04} is based on modular parameterisations of the underlying hypergeometric series.
This modular technique cannot be counted as elementary, and it does not cover certain important generalisations of
Ramanujan's identities such as recently discovered by the first author~\cite{Gu-on-wz,Gu-new,Gu7} formulae for $1/\pi^2$.
Guillera's approach to prove the latter is by using the powerful Wilf--Zeilberger (WZ) machinery, which itself is quite elementary in nature.
Finally, changes of variables allow us to ``translate'' the already derived identities into some harder ones~\cite{CZ,Ro,Zu1,Zu5};
for example, Bailey's cubic transformation \cite[eq.~(5.6)]{GS}
\begin{equation}
\label{bailey1}
\sum_{n=0}^{\infty}\frac{(\frac12)_n^3}{n!^3}\,x^n
=\frac2{(4-x)^{1/2}}\sum_{n=0}^{\infty}\frac{(\frac12)_n(\frac16)_n(\frac56)_n}{n!^3}\biggl(\frac{27x^2}{(4-x)^3}\biggr)^n
\end{equation}
with the choice $x=1/2^6$ translates the WZ provable equation~\eqref{rama-64} into the WZ resistant equation~\eqref{rama28}.
Note, however, that a starting point of any such translation has to be an already established one, by modular or WZ means.

The principal aim of this note is to demonstrate that in many, if not all, cases we may start from a
``semi-Ramanujan'' identity for $1/\pi$ (or for a similar quantity) whose existence only depends on the asymptotic behaviour of a solution
of a Picard--Fuchs equation at its singular point. Then translation brings the identity into ``traditional'' Ramanujan-type formulae.
Our approach is illustrated by several examples, in particular, we prove equations \eqref{rama-ex6}, \eqref{rama28} (hence \eqref{rama-64}
by translating via~\eqref{bailey1}), \eqref{e00}, \eqref{rama-10-1} and \eqref{rama7-4}, while equations \eqref{e03} and \eqref{e04}
(of degrees 37 and 29, respectively) remain beyond the reach of our techniques exposed here.
Ramanujan's and Ramanujan-type identities for $1/\pi$ we discuss in the paper are all classified according to
level and degree \cite{CC2,Coo}. We do not explicitly address the two characteristics (though the level does appear below)
as they are deeply related to modular (hence non-elementary!) parameterisations of the formulae. But it is worth mentioning
that all identities shown in the paper correspond to relatively low degrees. The principal difficulty in dealing with
high-degree Ramanujan-type identities is lack of efficient asymptotic techniques; see Example~\ref{ex4} and our speculative
derivation of the conjectural equation~\eqref{jg74} in Section~\ref{s5}.

The rest of the paper is organised as follows. In Section \ref{s2} we outline the method, while Sections \ref{s3} and \ref{s4}
are dedicated to several illustrative examples. Finally, Section~\ref{s5} highlights further directions of potential applicability
of the method, for example, for proving some experimentally observed formulae for $1/\pi^2$.

\section{Semi-Ramanujan identities}
\label{s2}

The classical theorem attributed to Stolz and Ces\`aro says that {\em if $b_n$ is increasing and unbounded
and
$$
\lim_{n\to\infty} \frac{a_{n+1}-a_n}{b_{n+1}-b_n} = \gamma,
$$
then
$$
\lim_{n\to\infty} \frac{a_n}{b_n} = \gamma
$$
as well}. Take as an example the $n$-th partial sums $a_n=\sum_{k=0}^nc_kr^k$ of an analytical solution $\sum_{n=0}^\infty c_nx^n$
to a linear differential equation, a solution which converges in the disc of radius $r$ where $r>0$
is a \emph{regular} singularity of the equation. If $c_n\sim\operatorname{const}\cdot r^{-n}n^{\alpha-1}$ for some $\alpha>0$,
then $\alpha$~is automatically a local exponent of the differential equation at $x=r$; assume that the exponent has multiplicity~1.
Then
$$
\sum_{n=0}^\infty c_nx^n=\gamma\Bigl(1-\frac xr\Bigr)^\alpha+\text{higher powers of }\Bigl(1-\frac xr\Bigr)
\qquad\text{as}\quad x\to r^-,
$$
and to find the leading coefficient $\gamma$ we can take $b_n=\sum_{k=0}^n(\alpha)_k/k!$,
the $n$-th partial sums of the power series of $(1-x/r)^{-\alpha}$ at $x=r$, and use the Stolz--Ces\`aro theorem:
\begin{equation}
\gamma=\lim_{x\to r^-}\frac{\sum_{n=0}^\infty c_nx^n}{(1-x/r)^{-\alpha}}=\lim_{n\to\infty}\frac{c_nr^n}{(\alpha)_n/n!};
\label{SC}
\end{equation}
the interchange of the limits is legitimate because the quotient on the left-hand side represents a continuous function
in the compact interval $0\le x\le r$. Below we typically choose $\alpha=\frac12$ in~\eqref{SC},
and our starting example, which has motivated the method in this paper and whose proof by means of the
argument above is due to A.~Mellit~\cite{Mel}, is as follows.

\begin{theorem}
\label{th1}
For $m$ a positive integer and $s_1,\dots,s_m$ in the interval $0<s<1$, the following limit is true:
\begin{equation}
\lim_{x\to 1^-} \sqrt{1-x} \sum_{n=0}^\infty \frac{(\frac12)_n (s_1)_n(1-s_1)_n\cdots (s_m)_n(1-s_m)_n}{n!^{2m+1}}n^m x^n
= \prod_{j=1}^m \frac{\sin\pi s_j}{\pi}.
\label{lim-hyper}
\end{equation}
\end{theorem}

\begin{proof}
Here we only use the simple fact
$$
\lim_{n\to\infty}\frac{(s)_n(1-s)_nn}{n!^2}
=\frac1{\Gamma(s)\,\Gamma(1-s)}=\frac{\sin\pi s}\pi
$$
which follows from Euler's definition
$$
\Gamma(s)=\lim_{n\to\infty}\frac{n!\,n^{s-1}}{(s)_n}
$$
of the gamma function.
\end{proof}

It is equation \eqref{lim-hyper} which we may call a semi-Ramanujan formula for $1/\pi$.
In his manuscript \cite{Ay} A.~Aycock, an undergraduate student of Johannes-Gutenberg-Universit\"at in Mainz,
uses some known Ramanujan-type series for $1/\pi$ and the translation method~\cite{Zu5} to prove
the limit in Theorem~\ref{th1} when $m=1$ and $s\in\{\frac12,\frac13,\frac14,\frac16\}$.
But because those instances are already known, even in a greater generality, we can\dots\ reverse Aycock's argument!

Let us do so, using Bailey's cubic transformation \cite[eq.~(5.3)]{GS}
\begin{equation}
\label{bailey2}
(1-4x)^{1/2}\sum_{n=0}^{\infty}\frac{(\frac12)_n^3}{n!^3}\,x^n
=\sum_{n=0}^{\infty}\frac{(\frac12)_n(\frac16)_n(\frac56)_n}{n!^3}\biggl(-\frac{27x}{(1-4x)^3}\biggr)^n
\end{equation}
which is a companion of~\eqref{bailey1}.

\begin{example}
\label{ex0}
Applying $\theta_x:=x\dfrac{\d}{\d x}$ to the both sides of~\eqref{bailey2} we arrive at
\begin{multline}
\label{bailey2a}
(1-4x)^{-1/2}\sum_{n=0}^{\infty}\frac{(\frac12)_n^3}{n!^3}\,\bigl(-2x+(1-4x)n\bigr)x^n
\\
=\frac{1+8x}{1-4x}\sum_{n=0}^{\infty}\frac{(\frac12)_n(\frac16)_n(\frac56)_n}{n!^3}\,n\biggl(-\frac{27x}{(1-4x)^3}\biggr)^n,
\end{multline}
where we let $x\to(-1/8)^+$. Note that for $y=y(x):=-27x/(1-4x)^3$ we have
$$
1-y=\frac{(1+8x)^2(1-x)}{(1-4x)^3},
$$
therefore
\begin{multline*}
\lim_{x\to(-1/8)^+}\frac{(1+8x)(1-x)^{1/2}}{(1-4x)^{3/2}}
\sum_{n=0}^{\infty}\frac{(\frac12)_n(\frac16)_n(\frac56)_n}{n!^3}\,n\biggl(-\frac{27x}{(1-4x)^3}\biggr)^n
\\
=\lim_{y\to1^-}(1-y)^{1/2}\sum_{n=0}^{\infty}\frac{(\frac12)_n(\frac16)_n(\frac56)_n}{n!^3}\,ny^n
=\frac1{2\pi}
\end{multline*}
by Theorem~\ref{th1}. Comparing the limit with the one for the left-hand side of~\eqref{bailey2a} we deduce
\begin{align*}
\frac1{2\pi}
&=\frac{(1-x)^{1/2}}{1-4x}\sum_{n=0}^{\infty}\frac{(\frac12)_n^3}{n!^3}\,\bigl(-2x+(1-4x)n\bigr)x^n\bigg|_{x=-1/8}
\\
&=\frac1{4\sqrt2}\sum_{n=0}^{\infty}\frac{(\frac12)_n^3}{n!^3}\,(1+6n)\biggl(-\frac18\biggr)^n,
\end{align*}
a formula for $1/\pi$ of Ramanujan type.
\qed
\end{example}

Let us try to formalise this approach.
Assume we are trying to prove a formula for $1/\pi$ (or another constant in general) attached to the point
$y=y_0$ and the generating function $\sum_{n=0}^\infty A_ny^n$, the latter satisfying an arithmetic differential equation
(usually of order~3). Then we look for an algebraic substitution $y=y(x)$ and accompanying algebraic multiple $g=g(x)$ such that
\begin{equation}
\label{tran}
\sum_{n=0}^\infty C_nx^n
=g(x)\sum_{n=0}^\infty A_ny(x)^n
\end{equation}
also satisfies an arithmetic differential equation, and $y_0=y(r)$ where $r$ is precisely the radius of convergence
of $\sum_{n=0}^\infty C_nx^n$. Suppose furthermore that we can then compute the limit
$$
\lim_{x\to r^-}\sqrt{1-x/r}\sum_{n=0}^\infty C_nnx^n
=\lim_{n\to\infty}\frac{C_nnr^n}{(\frac12)_n/n!}=\gamma,
$$
using Theorem~\ref{th1} or other means. Then a suitable combination $\sum_{n=0}^\infty A_n(a+bn)y_0^n$,
obtained by differentiation and taking the limit as $x\to r$ in~\eqref{tran},
is an algebraic multiple of~$\gamma$.

\smallskip
The series we deal with so far are special cases of the hypergeometric series
$$
{}_mF_{m-1}\biggl(\begin{matrix} a_1, \, \dots, \, a_m \\ b_1, \, \dots, \, b_{m-1} \end{matrix} \biggm|x\biggr)
=\sum_{n=0}^\infty\frac{(a_1)_n\dotsb(a_m)_n}{n!\,(b_1)_n\dotsb(b_{m-1})_n}\,x^n
$$
which converges in the disc $|x|<1$. If all the parameters $a_1,\dots,a_m$ and $b_1,\dots,b_{m-1}$ are positive,
the function defined by the series can be analytically continued to $\mathbb C\setminus[0,+\infty)$ via the Barnes integral,
\begin{multline}
\label{barnes}
{}_mF_{m-1}\biggl(\begin{matrix} a_1, \, \dots, \, a_m \\ b_1, \, \dots, \, b_{m-1} \end{matrix} \biggm|x\biggr)
=\frac{\Gamma(b_1)\dotsb\Gamma(b_{m-1})}{\Gamma(a_1)\dotsb\Gamma(a_m)}
\\ \times\frac1{2\pi i}
\int_{-a-i\infty}^{-a+i\infty}\frac{\Gamma(a_1+t)\dotsb\Gamma(a_m+t)\,\Gamma(-t)}{\Gamma(b_1+t)\dotsb\Gamma(b_{m-1}+t)}\,(-x)^t\,\d t,
\end{multline}
where the line of integration $\Re t=-a$ separates the poles $t=0,1,2,\dots$ of $\Gamma(-t)$ from those of
$\Gamma(a_1+t)\dotsb\Gamma(a_m+t)$; for example, we can take $a$ to be any in the range $0<a<\min\{a_1,\dots,a_m\}$.
With the hypergeometric function defined by~\eqref{barnes} in mind we have another (well-known!) family of semi-Ramanujan identities.

\begin{theorem}
\label{th2}
Let the parameters $a_1,\dots,a_m$ and $b_1,\dots,b_{m-1}$ of the hypergeometric function be positive
and $a_m<\min\{a_1,\dots,a_{m-1}\}$ the smallest one in the first group. Then
\[
\lim_{x\to-\infty} (-x)^{a_m}\,{}_mF_{m-1}\biggl(\begin{matrix} a_1, \, \dots, \, a_m \\ b_1, \, \dots, \, b_{m-1} \end{matrix} \biggm|x\biggr)
=\prod_{j=1}^{m-1}\frac{\Gamma(a_j-a_m)}{\Gamma(a_j)}\,\frac{\Gamma(b_j)}{\Gamma(b_j-a_m)}.
\]
\end{theorem}

\begin{proof}
By the residue theorem, the limit can be interpreted as the residue of
$$
\frac{\Gamma(b_1)\dotsb\Gamma(b_{m-1})}{\Gamma(a_1)\dotsb\Gamma(a_m)}\,
\frac{\Gamma(a_1+t)\dotsb\Gamma(a_m+t)\,\Gamma(-t)}{\Gamma(b_1+t)\dotsb\Gamma(b_{m-1}+t)}\,(-x)^{a_m+t}
$$
at its simple pole $t=-a_m$ (cf.\ \cite[eqs.~(2.3.12) and (2.4.6)]{AAR} for the $m=2$ case).
\end{proof}

\begin{example}
\label{ex6}
Write the transformation~\eqref{bailey2} in the hypergeometric form
\begin{equation}
\label{bailey2b}
(1-4x)^{1/2}{}_3F_2\biggl(\begin{matrix} \frac12, \, \frac12, \, \frac12 \\ 1, \, 1 \end{matrix} \biggm|x\biggr)
={}_3F_2\biggl(\begin{matrix} \frac12, \, \frac16, \, \frac56 \\ 1, \, 1 \end{matrix} \biggm|-\frac{27x}{(1-4x)^3}\biggr).
\end{equation}
By Theorem~\ref{th2} we have
$$
\lim_{y\to-\infty}(-y)^{1/6}{}_3F_2\biggl(\begin{matrix} \frac12, \, \frac16, \, \frac56 \\ 1, \, 1 \end{matrix} \biggm|y\biggr)
=\frac{2\sqrt\pi}{\sqrt3\,\Gamma(\frac56)^3}
$$
and
\begin{multline*}
\lim_{y\to-\infty}(-y)^{1/2}(1+6\theta_y){}_3F_2\biggl(\begin{matrix} \frac12, \, \frac16, \, \frac56 \\ 1, \, 1 \end{matrix} \biggm|y\biggr)
=\lim_{y\to-\infty}(-y)^{1/2}{}_3F_2\biggl(\begin{matrix} \frac12, \, \frac76, \, \frac56 \\ 1, \, 1 \end{matrix} \biggm|y\biggr)
=\frac{2\sqrt3}\pi,
\end{multline*}
where $\theta_y:=y\dfrac{\d}{\d y}$.
By making the substitution $y=-27x/(1-4x)^3$, translating them to the left-hand side of \eqref{bailey2b} and taking the limit $x\to(1/4)^-$ we obtain
\begin{equation}
\label{bailey2c}
\sum_{n=0}^{\infty}\frac{(\frac12)_n^3}{n!^3}\,\frac{1}{4^n}
=\frac{2\sqrt[3]{2}}{3}\,\frac{\sqrt{\pi}}{\Gamma(\frac56)^3}
\end{equation}
and Ramanujan's identity~\eqref{rama-ex6}, respectively.
\qed
\end{example}

The algebraic transformations we use below are all of modular origin. However, as soon as they are given one can verify
them directly without any reference to the underlying modular functions. This makes our strategy ``elementary'' in the sense
``modular free.''

\section(Examples of Ramanujan's formulae for 1/\003\300){Examples of Ramanujan's formulae for $1/\pi$}
\label{s3}

\begin{example}
\label{ex1}
To prove Ramanujan's identities \eqref{e00} and \eqref{rama7-4} we use the transformation~\cite[eq.~(3.7)]{Ro}
\begin{equation}\label{rogers}
{}_3F_2\biggl(\begin{matrix} \frac12, \, \frac14, \, \frac34 \\ 1, \, 1 \end{matrix} \biggm|\frac{256k}{9(1+3k)^4}\biggr)
=\frac{1+3k}{1+k/3}\,{}_3F_2\biggl(\begin{matrix} \frac12, \, \frac14, \, \frac34 \\ 1, \, 1 \end{matrix} \biggm|\frac{256k^3}{9(3+k)^4}\biggr).
\end{equation}
Apply $k\dfrac{\d}{\d k}$ to the both sides of \eqref{rogers} and pass to the limit $k\to(1/9)^-$ using Theorem~\ref{th1}
with $x=256k/(9(1+3k)^4)$ on the left-hand side; the result is equation \eqref{rama7-4}. Similarly, equation \eqref{e00} follows
from application of Theorem~\ref{th2} to the function
\begin{equation}\label{F123}
(1+4\theta_x){}_3F_2\biggl(\begin{matrix} \frac12, \, \frac14, \, \frac34 \\ 1, \, 1 \end{matrix} \biggm|x\biggr)
={}_3F_2\biggl(\begin{matrix} \frac12, \, \frac54, \, \frac34 \\ 1, \, 1 \end{matrix} \biggm|x\biggr)
\end{equation}
and then letting $k\to(-1/3)^+$. We also have a side identity like in Example~\ref{ex6} from application of Theorem~\ref{th2}
directly to the $_3F_2$ hypergeometric function on the left-hand side of~\eqref{rogers}; namely,
\[
\sum_{n=0}^{\infty}\frac{(\frac12)_n(\frac14)_n(\frac34)_n}{n!^3} \biggl(-\frac{1}{48}\biggr)^n
=\frac{2^{3/2}\pi}{3^{5/4}\Gamma(\frac34)^4}.
\]

In a similar way, from the limiting cases $z\to(1/2)^-$ and $z\to(9/8)^-$ of the Kummer--Goursat transformation
(the square of~\cite[eq.~(25)]{Vid})
\begin{equation}\label{goursat}
{}_3F_2\biggl(\begin{matrix} \frac12, \, \frac13, \, \frac23 \\ 1, \, 1 \end{matrix} \biggm|4z(1-z)\biggr)
=\frac3{\sqrt{9-8z}}\,{}_3F_2\biggl(\begin{matrix} \frac12, \, \frac16, \, \frac56 \\ 1, \, 1 \end{matrix} \biggm|\frac{64z^3(1-z)}{(9-8z)^3}\biggr)
\end{equation}
we produce the Ramanujan-type formulae
\begin{equation*}
\sum_{n=0}^{\infty}\frac{(\frac12)_n(\frac16)_n(\frac56)_n}{n!^3}\,(1+11n)\biggl(\frac{4}{125}\biggr)^n=\frac{5\sqrt5}{2\pi\sqrt3}
\end{equation*}
and
\begin{equation*}
\sum_{n=0}^{\infty}\frac{(\frac12)_n(\frac13)_n(\frac23)_n}{n!^3}\,(1+5n)\biggl(-\frac{9}{16}\biggr)^n=\frac{4}{\pi\sqrt3},
\end{equation*}
respectively.
\qed
\end{example}

\begin{example}
\label{ex2}
In many cases, the leading term in the asymptotics of a binomial sum expression can be obtained by using Stirling's
formula and other standard tools of asymptotic analysis. For example, the growth of the sequence
$C_n=\sum_{j=0}^{n}\binom{n}{j}^4$ is determined from \cite[Theorem~1]{McI}:
$C_n\sim2^{4n+1/2}(\pi n)^{-3/2}$ as $n\to\infty$. On using the Stolz--Ces\`aro theorem we obtain
\begin{equation}\label{limit-yy}
\lim_{x\to(1/16)^{-}} \sqrt{1-16x}\sum_{n=0}^{\infty}C_nnx^n
=\lim_{n\to\infty}\frac{C_nn/16^n}{(\frac12)_n/n!}
=\frac{\sqrt{2}}{\pi}.
\end{equation}
On the other hand, the generating function $\sum_{n=0}^\infty C_nx^n$ can be written hypergeometrically;
namely, applying \cite[formula (2-53)]{LR}, Clausen's formula and \cite[Theorem 3.6]{Coo0} (see also \cite{Coo2})
we have
\begin{align}
&
\sum_{n=0}^{\infty} C_n\biggl(\frac{k(1-k^2)(1+k-k^2)(1-4k-k^2)}{(1+k^2)^4}\biggr)^n
\nonumber\\ &\qquad
=\frac{(1+k^2)^2}{1+22k-6k^2-22k^3+k^4}
\nonumber\\ &\qquad\quad\times
{}_3F_2\biggl(\begin{matrix} \frac12, \, \frac14, \, \frac34 \\ 1, \, 1 \end{matrix} \biggm|
\frac{256k(1-k^2)(1+k-k^2)^5(1-4k-k^2)^5}{(1+k^2)^4(1+22k-6k^2-22k^3+k^4)^4}\biggr).
\label{transfor-zu}
\end{align}
Apply $k\dfrac{\d}{\d k}$ to the both sides of \eqref{transfor-zu} and compute the limit as
$k\to(\sqrt{10+4\sqrt5}-2-\sqrt5)^-$ which corresponds to $x\to(1/16)^-$ in~\eqref{limit-yy}
on the left-hand side of the resulting identity. This produces the first ``elementary'' proof
of Ramanujan's identity~\eqref{rama-10-1}, as no translation of a WZ provable identity is known for it.

Another limit as
$$
k\to\biggl(\frac{11+5\sqrt5-\sqrt{250+110\sqrt5}}2\biggr)^+
$$
corresponds to taking
$$
y(k):=\frac{256k(1-k^2)(1+k-k^2)^5(1-4k-k^2)^5}{(1+k^2)^4(1+22k-6k^2-22k^3+k^4)^4}\to-\infty.
$$
Therefore, combining Theorem~\ref{th2} applied to the function \eqref{F123} with the transformation~\eqref{transfor-zu}
we deduce the Ramanujan-type formula
$$
\sum_{n=0}^{\infty}\sum_{j=0}^n\binom{n}{j}^4\,(1+3n)\biggl(-\frac1{20}\biggr)^n
=\frac5{2\pi}
$$
given previously in \cite[Theorem~5.2]{Coo2}. Applying Theorem~\ref{th2} directly to the $_3F_2$ series
on the right-hand side of~\eqref{transfor-zu} we obtain the side evaluation
$$
\sum_{n=0}^{\infty}\sum_{j=0}^n\binom{n}{j}^4\biggl(-\frac1{20}\biggr)^n
=\frac{\pi}{5^{1/4}\Gamma(\frac34)^4}
$$
which seems to be new.
\qed
\end{example}

\begin{example}
\label{ex3}
This time we deal with the level~7 sequence~\cite{Coo}
$$
C_n=\sum_{j=0}^{n}\binom{n}{j}^2\binom{2j}{n}\binom{n+j}{j}
\sim\frac{27^{n+1/2}}4\,(\pi n)^{-3/2} \quad \text{as $n\to\infty$},
$$
so that
\begin{equation}\label{limit-yy28}
\lim_{x\to(1/27)^{-}} \sqrt{1-27x}\sum_{n=0}^{\infty}C_nnx^n
=\lim_{n\to\infty}\frac{C_nn/27^n}{(\frac12)_n/n!}
=\frac{3\sqrt{3}}{4\pi}
\end{equation}
by the Stolz--Ces\`aro theorem. It follows from \cite[Theorem~3.1]{Coo} and \cite[Lemmas 4.1 and 4.3]{CC1} that
\begin{multline}
\label{id7}
\sum_{n=0}^\infty C_n\biggl(\frac h{1+13h+49h^2}\biggr)^n
=\frac{\sqrt{1+13h+49h^2}}{\sqrt{1+5h+h^2}}
\\ \times
\sum_{n=0}^\infty\frac{(\frac16)_n(\frac12)_n(\frac56)_n}{n!^3}
\biggl(\frac{1728h^7}{(1+5h+h^2)^3(1+13h+49h^2)}\biggr)^n.
\end{multline}
Apply $h\dfrac{\d}{\d h}$ to the both sides of this transformation and compute the limit as
$h\to(1/7)^-$ which corresponds to $x\to(1/27)^-$ in~\eqref{limit-yy28}
on the left-hand side. The result is Ramanujan's series \eqref{rama28} which, in turn, is equivalent to~\eqref{rama-64}.
\qed
\end{example}

\begin{example}
\label{ex4}
Let $\sum_{n=0}^\infty A_nx^n$, with $A_0=1$, be the holomorphic solution of the differential equation $\mathcal Dy=0$, where
\begin{align*}
\mathcal D&:=\theta^3-x(2\theta+1)(11\theta^2+11\theta+5)
+x^2(\theta+1)(121\theta^2+242\theta+141)
\\ &\qquad
-98x^3(\theta+1)(\theta+2)(2\theta+3)
\end{align*}
and $\theta=\theta_x=x\dfrac{\d}{\d x}$. The equation is equivalent to a difference equation with polynomial in $n$
coefficients, for the sequence $A_n$ itself, and the Birkhoff--Trjitzinsky method~\cite{WiZe,Ze} allows us to determine
asymptotic behaviour of the latter:
$$
A_n\sim\frac{C(1+2\sqrt2)^{2n}}{n^{3/2}}\,\biggl(1+\frac{c_1}n+\frac{c_2}{n^2}+\dotsb\biggr) \quad\text{as $n\to\infty$},
$$
with the numerical value $C=0.44846956652386\dots$ which can be recognised as
$$
C=\frac{\sqrt{184+11\sqrt2}}{4\sqrt2\pi^{3/2}}.
$$
(An efficient strategy to compute the number is outlined in \cite[Section~3]{Za}.)

As D.~Zeilberger writes in~\cite{Ze}, ``The Birkhoff--Trjitzinsky method suffers from one drawback.
It only does the asymptotics up to a multiplicative constant $C$. But nowadays this is hardly a problem.
Just crank-out the first ten thousand terms of the sequence using the very recurrence whose asymptotics you are trying to find,
not forgetting to furnish the few necessary initial conditions, and then estimate the constant empirically.
If you are lucky, then Maple can recognize it in terms of `famous' constants like $e$ and $\pi$, by typing \texttt{identify(C);}.''

In order to not only discover the constant $C$ numerically but also to prove it rigorously, we need more information
about the sequence $A_n$, for example, a binomial sum expression for its terms.

The above asymptotics is equivalent to
\begin{equation}\label{lim-yy14}
\lim_{x\to r^-} \sqrt{(1-4x)(1-18x+49x^2)}\sum_{n=0}^\infty A_nnx^n = \frac{\sqrt{14}}{2 \pi},
\end{equation}
where $r=1/(1+2\sqrt2)^2$, through the Stolz--Ces\`aro theorem. The underlying modular structure (of level~14)
allows us to link the series $\sum_{n=0}^\infty A_nx^n$ to the (level~7) series $\sum_{n=0}^\infty C_nx^n$ from Example~\ref{ex3}:
\begin{equation}
\frac1{(1+v)(1+8v)}\sum_{n=0}^\infty A_n\biggl(\frac v{(1+v)(1+8v)}\biggr)^n
=\frac1{(1+4v)^2}\sum_{n=0}^\infty C_n\biggl(\frac v{(1+4v)^3}\biggr)^n.
\label{id14}
\end{equation}
The value $r$ corresponds to the choice $v=2^{-3/2}$, so that equation~\eqref{lim-yy14} translates into the formula
\begin{equation*}
\sum_{n=0}^\infty\sum_{j=0}^{n}\binom{n}{j}^2\binom{2j}{n}\binom{n+j}{j}
\bigl(8\sqrt2-11+(13\sqrt2-17)n\bigr)
\biggl(\frac{\sqrt2-1}{\sqrt2}\biggr)^{3n}
=\frac{\sqrt{16\sqrt2+13}}{\pi\sqrt7}
\end{equation*}
of Ramanujan type; it can be also derived from the results in \cite[Section~4]{Coo}
(it corresponds to taking $N=2$, $q=e^{-2\pi\sqrt{2/7}}$, $y_7=(1-1/\sqrt2)^3$ and $\lambda=(3-\sqrt2)/7$
in the notation in~\cite{Coo}). With the help of~\eqref{id14}, \eqref{id7}
and the known algebraic transformations of $_3F_2$ hypergeometric series, like \eqref{bailey2} and \eqref{goursat} above,
we can obtain some more Ramanujan-type formulae for $1/\pi$; for example,
\begin{equation*}
\sum_{n=0}^{\infty}\frac{(\frac12)_n(\frac14)_n(\frac34)_n}{n!^3}
\,(9-3\sqrt2+56n)\biggl(\frac{1}{11+8\sqrt2}\biggr)^{2n}
=\frac{2\sqrt{25+22\sqrt2}}{\pi}.
\end{equation*}
The latter identity can be found in \cite[eq.~(4.75)]{BCL}.
\qed
\end{example}

Here seems to be a right place to briefly comment about underlying modular parameterisations.
Any known generating function $\sum_{n=0}^\infty A_nx^n$ which gives rise to Ramanujan-type series for $1/\pi$
is annihilated by a third order linear differential equation $\mathcal D=\theta^3+\dotsb\in\mathbb Q[\theta](x)$,
at least in the genus~0 cases as defined by Conway and Norton~\cite{CN}. Define $P(x)\in\mathbb Q(x)$
to be the formal limit of $\mathcal D/\theta^3$ as $\theta\to\infty$ (see also \cite[Section~2]{AlGu2}); it is the polynomial
\[
1-2\cdot11x+121x^2-2\cdot98x^3
=(1-4x)(1-18x+49x^2)
\]
in Example~\ref{ex4} above. Then the Conway--Norton modular parameterisation~\cite[Table~3]{CN} $x(\tau)=1/J_G(\tau)$
and the use of the Stolz--Ces\`aro theorem result in
\begin{equation}\label{limit-general}
\lim_{x \to r^-} \sqrt{P(x)} \sum_{n=0}^{\infty} A_n n x^n = \frac{\sqrt{\ell}}{2\pi}
\end{equation}
where $\ell$ is the level of the corresponding modular group~$G$ and $r$ is a suitable (usually, the smallest) zero
of~$P(x)$. This can be counted as an arithmetic generalisation of the $m=1$ case of Theorem~\ref{th1}; a heuristic
analogue of \eqref{limit-general} exists for $m=2$ as well~\cite{AlGu2}. As hinted in~\cite{Zu5},
the semi-Ramanujan identity~\eqref{limit-general} can be further translated (at least, in theory) into other Ramanujan-type formulae
for $1/\pi$ whenever the modular parameter $\tau$ is chosen from the same quadratic field $\mathbb Q(\sqrt{-\ell})$.

\section{Multi-parameter translation}
\label{s4}

\begin{example}
\label{ex7}
Begin with an identity due to Whipple \cite[eq.~(2.2)]{GS}:
\begin{multline*}
(1-x)^{1/2+p}\,{}_3F_2\biggl(\begin{matrix} \frac12+p, \, \frac12+q, \, \frac12+r \\ 1+p-q, \, 1+p-r \end{matrix} \biggm|x\biggr)
\\
={}_3F_2\biggl(\begin{matrix} \frac12+p-q-r, \, \frac14+\frac12p, \, \frac34+\frac12p \\ 1+p-q, \, 1+p-r \end{matrix} \biggm|-\frac{4x}{(1-x)^2}\biggr).
\end{multline*}
Then apply the operator $\theta_x$ and take the limit as $x\to(-1)^{+}$ using Theorem~\ref{th1} on the right-hand side,
to obtain the identity
\begin{multline*}
\sum_{n=0}^{\infty}\frac{(\frac12+p)_n(\frac12+q)_n(\frac12+r)_n}{n!\,(1+p-q)_n(1+p-r)_n}\,(1+2p+4n)\,(-1)^n
\\
=\frac{2^{1-p}}{\pi}\,\frac{(1)_{p-q}(1)_{p-r}}{(\frac12)_{p-q-r}(\frac14)_{p/2}(\frac34)_{p/2}}
=\frac{2\Gamma(1+p-q)\Gamma(1+p-r)}{\Gamma(\frac12+p-q-r)\Gamma(\frac12+p)}.
\end{multline*}
This generalisation was first obtained in \cite{Chu} using a different method.

In a similar way, the limiting $x\to(-1/8)^+$ case of Bailey's identity \cite[eq.~(5.3)]{GS}
\begin{multline}\label{Gessel-St1}
(1-4x)^{1/2+2p}\,{}_3F_2\biggl(\begin{matrix} \frac12+2p, \, \frac12-2q, \, \frac12+2q \\ 1+p-q, \, 1+p+q \end{matrix} \biggm|x\biggr)
\\
={}_3F_2\biggl(\begin{matrix} \frac12+\frac23p, \, \frac16+\frac23p, \, \frac56+\frac23p \\ 1+p-q, \, 1+p+q \end{matrix} \biggm|
-\frac{27x}{(1-4x)^3}\biggr),
\end{multline}
which is a general form of~\eqref{bailey2},
leads us to the following generalised Ramanujan-type series:
\begin{multline*}
\sum_{n=0}^\infty\frac{(\frac12+2p)_n(\frac12-2q)_n(\frac12+2q)_n}{n!\,(1+p-q)_n(1+p+q)_n(1)_n}
\,(1+4p+6n)\biggl(-\frac18\biggr)^n
\\
=\frac{2^{5/2+2p}\pi^{1/2}}{3^{2p}}\,
\frac{\Gamma(1+p-q)\Gamma(1+p+q)}{\Gamma(\frac12+\frac23p)\Gamma(\frac16+\frac23p)\Gamma(\frac56+\frac23p)},
\end{multline*}
which was first obtained in \cite{ChDi} in a different way.

If instead we compute the limit as $x\to(1/4)^{+}$ of \eqref{Gessel-St1} using Theorem~\ref{th2} on the right-hand side,
we find that
\begin{multline*}
\sum_{n=0}^\infty\frac{(\frac12+2p)_n(\frac12-2q)_n(\frac12+2q)_n}{n!\,(1+p-q)_n(1+p+q)_n}\,\frac{1}{4^n}
\\
=\frac{2^{4(1+p)/3}\pi}{3^{1+2p}}\,
\frac{\Gamma(1+p-q)\Gamma(1+p+q)}
{\Gamma(\frac12+\frac23p)\Gamma(\frac56+\frac23p)\Gamma(\frac56+\frac13p-q)\Gamma(\frac56+\frac13p+q)}.
\end{multline*}
which generalises equation~\eqref{bailey2c}.
\qed
\end{example}

Due to the presence of extra parameters, the identities in Example~\ref{ex7} can be directly shown by the Wilf--Zeilberger method.

\begin{example}
\label{ex5}
Take
\[
A_n = \frac{(\frac14)_n(\frac34)_n}{n!^2}\sum_{j=0}^n\frac{(\frac12)_j^3(\frac12)_{n-j}}{j!^3(n-j)!}
\]
and consider the related transformation \cite[Theorem~2]{Zu1}
\[
\sum_{n=0}^{\infty}A_n\biggl(-\frac{4x}{(1-x)^2}\biggr)^n
=\sqrt{1-x}\sum_{n=0}^\infty\frac{(\frac12)_n^5}{n!^5}\,x^n.
\]
$\theta_x$-Differentiating and taking the limit as $x\to(-1)^+$, we recover the identity
\[
\sum_{n=0}^\infty\frac{(\frac12)_n^5}{n!^5}\,(1+4n)\,(-1)^n
=\frac{2}{\Gamma(\frac34)^4},
\]
which is due to Ramanujan. The general case of the transformation \cite[Theorem~1]{Zu1}
results in an identity also recorded by Ramanujan \cite[Entry~31, p.~41]{Be2} which, in turn,
is a special case of Whipple's transformation.
\qed
\end{example}

\section{Conclusion}
\label{s5}

It seems plausible that a suitable generalisation of the translation techniques considered in the paper,
maybe an extension of a purely computational nature, will make it possible to prove some experimentally observed
but not yet proven formulae, like
\begin{align}
\sum_{n=0}^\infty\frac{(\frac12)_n(\frac18)_n(\frac38)_n(\frac58)_n(\frac78)_n}{n!^5}\,
(15+304n+1920n^2)\,\frac1{7^{4n}}
&=\frac{56\sqrt7}{\pi^2},
\label{jg74}
\\
\sum_{n=0}^\infty\frac{(\frac12)_n(\frac14)_n(\frac34)_n(\frac13)_n(\frac23)_n}{n!^5}\,
(5+63n+252n^2)\biggl(-\frac{1}{48}\biggr)^n&=\frac{48}{\pi^2},
\nonumber\\
\sum_{n=0}^\infty\frac{(\frac12)_n^7}{n!^7}\,(1+14n+76n^2+168n^3)\,\frac1{2^{6n}}
&=\frac{32}{\pi^3}.
\nonumber
\end{align}
Note the resemblance of these equations with \eqref{rama7-4}, \eqref{e00} and \eqref{rama-64}, respectively,
and the fact that for each of the latter we already know ``suitable'' transformations from Examples~\ref{ex1}
and~\ref{ex3}. For instance, the proof of~\eqref{jg74} might go as follows.

Define the sequence $C_n$ by the equation
\begin{multline}\label{n-tr}
\sum_{n=0}^\infty C_n\biggl(\frac{256k}{9(1+3k)^4}\biggr)^n
\\
=\frac{\sqrt{1+3k}}{\sqrt{1+k/3}}
\sum_{n=0}^\infty\frac{(\frac12)_n(\frac18)_n(\frac38)_n(\frac58)_n(\frac78)_n}{n!^5}
\,(5+(30-54k)n)\biggl(\frac{256k^3}{9(3+k)^4}\biggr)^n.
\end{multline}
Then the generating function $\sum_{n=0}^\infty C_nx^n$ satisfies a linear differential equation, equivalently,
the sequence $C_n$ satisfies a difference equation with polynomial coefficients.
Assuming we can guess from the equation a hypergeometric (binomial) expression for the terms $C_n$,
we can find the asymptotics
$$
C_n\sim2\sqrt6\pi^{-5/2}n^{-3/2} \quad\text{as $n\to\infty$},
$$
so that
$$
\lim_{x\to1^-}\sqrt{1-x}\sum_{n=0}^\infty C_nnx^n
=\frac{2\sqrt6}{\pi^2}
$$
by the Stolz--Ces\`aro theorem. Applying $k\dfrac{\d}{\d k}$ to the both sides of \eqref{n-tr} and
computing the limit as $k\to(1/9)^-$ give us the desired identity \eqref{jg74}.

\smallskip
Another aspect of the method is its potential applicability to proving so-called Ramanujan-type
supercongruences~\cite{Gu6,GuiZu,Zu0}. This belief is supported by some recent results of L.~Long and
her collaborators \cite{KLMSY,Long}.

\section*{Acknowledgements}
We are grateful to Alexander Aycock and Anton Mellit for original, though implicit, inspiration to launch the project.
Special thanks go to Gert Almkvist, Heng Huat Chan, Shaun Cooper, Arne Meurman and James Wan for fruitful conversations on parts of this work.
Cooper's comments significantly helped us to improve the current presentation.


\begin{thebibliography}{99}


\bibitem{AlGu2}
\textsc{G. Almkvist} and \textsc{J. Guillera},
Ramanujan--Sato-like series,
in ``Number Theory and Related Fields, In Memory of Alf van der Poorten'',
J.\,M.~Borwein et~al. (eds.),
\emph{Springer Proc. Math. Stat.} \textbf{43} (2013), 16~pages;
\emph{Preprint} \texttt{arXiv:\,1201.5233 [math.NT]} (2012).

\bibitem{AAR}
\textsc{G.\,E.~Andrews}, \textsc{R.~Askey} and \textsc{R.~Roy},
\emph{Special functions},
Encyclopedia Math. Appl. \textbf{71}
(Cambridge University Press, Cambridge, 1999).

\bibitem{Ay}
\textsc{A.~Aycock},
On the probably simplest method of proving formulas for $1/\pi$,
\emph{An unpublished manuscript} (24 February 2012), 50~pages.

\bibitem{BaBeCh}
\textsc{N.\,D. Baruah}, \textsc{B.\,C. Berndt} and \textsc{H.\,H. Chan},
Ramanujan's series for $1/\pi$: a survey,
\emph{Amer. Math. Monthly} \textbf{116} (2009), 567--587.

\bibitem{Be2}
\textsc{B.\,C.~Berndt},
\emph{Ramanujan's Notebooks, Part II}
(Springer-Verlag, New York, 1989).


\bibitem{BCL}
\textsc{B.\,C. Berndt}, \textsc{H.\,H. Chan} and \textsc{W.-C. Liaw},
On Ramanujan's quartic theory of elliptic functions,
\emph{J. Number Theory} \textbf{88}:1 (2001), 129--156.

\bibitem{Bo}
\textsc{J.\,M. Borwein} and \textsc{P.\,B. Borwein},
\emph{Pi and the AGM: A Study in Analytic Number Theory and Computational Complexity},
Canad. Math. Soc. Series Monographs Advanced Texts (John Wiley, New York, 1987).



\bibitem{CC1}
\textsc{H.\,H.~Chan} and \textsc{S.~Cooper},
Eisenstein series and theta functions to the septic base,
\emph{J. Number Theory} \textbf{128}:3 (2008), 680--699.

\bibitem{CC2}
\textsc{H.\,H.~Chan} and \textsc{S.~Cooper},
Rational analogues of Ramanujan's series for $1/\pi$,
\emph{Math. Proc. Cambridge Philos. Soc.} \textbf{153}:2 (2012), 361--383.

\bibitem{CZ}
\textsc{H.\,H.~Chan} and \textsc{W.~Zudilin},
New representations for Ap\'ery-like sequences,
\emph{Mathematika} \textbf{56} (2010), 107--117.

\bibitem{Chu}
\textsc{W. Chu},
Dougall's bilateral $_2H_2$ series and Ramanujan-like $\pi$ formulas,
\emph{Math. Comp.} \textbf{276} (2011), 2223--2251.

\bibitem{Chud}
\textsc{D.\,V.~Chudnovsky} and \textsc{G.\,V.~Chudnovsky},
Approximations and complex multiplication according to Ramanujan,
in ``Ramanujan revisited'', Urbana-Champaign, Ill. 1987
(Academic Press, Boston, MA 1988), 375--472.

\bibitem{CN}
\textsc{J.\,H.~Conway} and \textsc{S.\,P.~Norton},
Monstrous moonshine,
\emph{Bull. London Math. Soc.} \textbf{11}:3 (1979), 308--339.

\bibitem{Coo0}
\textsc{S.~Cooper},
On Ramanujan's function $k(q)=r(q)r^2(q^2)$,
\emph{Ramanujan J.} \textbf{20} (2009), 311--328.

\bibitem{Coo2}
\textsc{S. Cooper},
Level $10$ analogues of Ramanujan's series for $1/\pi$,
\emph{J. Ramanujan Math. Soc.} \textbf{27} (2012), 59--76.

\bibitem{Coo}
\textsc{S.~Cooper},
Sporadic sequences, modular forms and new series for $1/\pi$,
\emph{Ramanujan J.} \textbf{29}:1-3 (2012), 163--183.

\bibitem{GS}
\textsc{I.~Gessel} and \textsc{D. Stanton},
Strange evaluations of hypergeometric series,
\emph{SIAM J. Math. Anal.} \textbf{13}:2 (1982), 295--308.


\bibitem{Gu-on-wz}
\textsc{J. Guillera},
On WZ-pairs which prove Ramanujan series,
\emph{Ramanujan J.} \textbf{22} (2008), 249--259.

\bibitem{Gu-new}
\textsc{J. Guillera},
A new Ramanujan-like series for $1/\pi^2$,
\emph{Ramanujan J.}  \textbf{26} (2011), 369--374.

\bibitem{Gu7}
\textsc{J.~Guillera},
Collection of Ramanujan-like series for $1/\pi^2$,
\emph{An unpublished manuscript}, available from the author's web site.

\bibitem{Gu6}
\textsc{J.~Guillera},
Mosaic supercongruences of Ramanujan-type,
\emph{Exp. Math.} \textbf{21} (2012), 65--68.


\bibitem{GuiZu}
\textsc{J.~Guillera} and \textsc{W.~Zudilin},
``Divergent'' Ramanujan-type supercongruences,
\emph{Proc. Amer. Math. Soc.} \textbf{140} (2012), 765--777.


\bibitem{KLMSY}
\textsc{J. Kibelbek}, \textsc{L. Long}, \textsc{K. Moss}, \textsc{B. Sheller} and \textsc{H. Yuan},
Supercongruences and complex multiplication,
\emph{Preprint} \texttt{arXiv:\,1210.4489 [math.NT]} (2012).

\bibitem{LR}
\textsc{M.\,N.~Lal\'\i n} and \textsc{M.\,D.~Rogers},
Functional equations for Mahler measures of genus-one curves,
\emph{Algebra and Number Theory} \textbf{1} (2007), 87--117.

\bibitem{Long}
\textsc{L. Long},
Hypergeometric evaluation identities and supercongruences,
\emph{Pacific J. Math.} \textbf{249}:2 (2011), 405--418.

\bibitem{McI}
\textsc{R.~McIntosh},
An asymptotic formula for binomial sums,
\emph{J. Number Theory} \textbf{58} (1996), 158--172.

\bibitem{Mel}
\textsc{A.~Mellit},
\emph{A personal communication to W.~Zudilin} (19 July 2011).

\bibitem{Ra}
\textsc{S. Ramanujan},
Modular equations and approximations to $\pi$,
\emph{Quart. J. Math.} \textbf{45} (1914), 350--372.

\bibitem{Ro}
\textsc{M. Rogers},
New $_5F_4$ hypergeometric transformations, three-variable Mahler measures, and formulas for $1/\pi$,
\emph{Ramanujan J.} \textbf{18} (2009), 327--340.

\bibitem{Vid}
\textsc{R.~Vid\=unas},
Algebraic transformations of Gauss hypergeometric functions,
\emph{Funkcial. Ekvac.} \textbf{52} (2009), 139--180.


\bibitem{ChDi}
\textsc{C.~Wei} and \textsc{D.~Gong},
Extensions of Ramanujan's two formulas for $1/\pi$,
\emph{Preprint} \texttt{arXiv:\,1202.1029 [math.CO]} (2012).

\bibitem{WiZe}
\textsc{J.~Wimp} and \textsc{D.~Zeilberger},
Resurrecting the asymptotics of linear recurrences,
\emph{J. Math. Anal. Appl.} \textbf{111}:1 (1985), 162--176.


\bibitem{Za}
\textsc{D.~Zagier},
Vassiliev invariants and a strange identity related to the Dedekind eta-function,
\emph{Topology} \textbf{40} (2001), 945--960.

\bibitem{Ze}
\textsc{D.~Zeilberger},
AsyRec: A Maple package for computing the asymptotics of solutions of linear recurrence equations with polynomial coefficients,
\emph{Personal Journal of Ekhad and Zeilberger} (4 April 2008), 2~pp.;
available from \texttt{http://www.math.rutgers.edu/\allowbreak\~{}zeilberg/\allowbreak mamarim/\allowbreak mamarimPDF/\allowbreak asy.pdf}.

\bibitem{Zu1}
\textsc{W.~Zudilin},
Quadratic transformations and Guillera's formulae for $1/\pi^2$,
\emph{Mat. Zametki} \textbf{81}:3 (2007), 335--340;
English transl.,
\emph{Math. Notes} \textbf{81}:3 (2007), 297--301.

\bibitem{Zu}
\textsc{W. Zudilin},
Ramanujan-type formulae for $1/\pi$: A second wind?,
in ``Modular Forms and String Duality'' (Banff, June 3--8, 2006), N.~Yui, H.~Verrill and C.\,F.~Doran (eds.),
\emph{Fields Inst. Commun. Ser.} \textbf{54} (Amer. Math. Soc., Providence, RI 2008), 179--188.

\bibitem{Zu0}
\textsc{W. Zudilin},
Ramanujan-type supercongruences,
\emph{J. Number Theory} \textbf{129}:8 (2009), 1848--1857.


\bibitem{Zu5}
\textsc{W.~Zudilin},
Lost in translation,
in ``Proceedings of the Waterloo Workshop in Computer Algebra (W80) (May 2011), In honour of Herbert S.~Wilf'',
I.~Kotsireas and E.\,V.~Zima (eds.),
\emph{Springer Proc. Math. Stat.} (to appear), 6~pages;
\emph{Preprint} \texttt{arXiv:\,1210.0269 [math.NT]} (2012).

\end{thebibliography}
\end{document}